\documentclass[11pt]{article}
\usepackage[margin=1in]{geometry} 
\usepackage{amssymb,amsfonts,amsmath,bbm,mathrsfs,stmaryrd}
\usepackage{xcolor}
\usepackage{url}

\usepackage{enumerate}

\usepackage{hyperref}
\hypersetup{colorlinks,
             linkcolor=black!75!red,
             citecolor=blue,
             pdftitle={},
             pdfproducer={pdfLaTeX},
             pdfpagemode=None,
             bookmarksopen=true
             bookmarksnumbered=true}

\usepackage{tikz}
\usetikzlibrary{arrows,calc,decorations.pathreplacing,decorations.markings,intersections,shapes.geometric,through,fit,shapes.symbols,positioning,decorations.pathmorphing}

\usepackage{braket}

\usepackage[amsmath,thmmarks,hyperref]{ntheorem}
\usepackage{cleveref}

\usepackage{babel}

\creflabelformat{enumi}{#2(#1)#3}

\crefname{section}{Section}{Sections}
\crefformat{section}{#2Section~#1#3} 
\Crefformat{section}{#2Section~#1#3} 

\crefname{subsection}{\S}{\S\S}
\AtBeginDocument{%
  \crefformat{subsection}{#2\S#1#3}%
  \Crefformat{subsection}{#2\S#1#3}%
}

\crefname{subsubsection}{\S}{\S\S}
\AtBeginDocument{%
  \crefformat{subsubsection}{#2\S#1#3}%
  \Crefformat{subsubsection}{#2\S#1#3}%
}

\theoremstyle{plain}

\newtheorem{lemma}{Lemma}[section]
\newtheorem{proposition}[lemma]{Proposition}
\newtheorem{corollary}[lemma]{Corollary}
\newtheorem{theorem}[lemma]{Theorem}

\theoremstyle{nonumberplain}
\newtheorem{theoremN}{Theorem}

\theoremstyle{plain}
\theorembodyfont{\upshape}
\theoremsymbol{\ensuremath{\blacklozenge}}

\newtheorem{definition}[lemma]{Definition}
\newtheorem{example}[lemma]{Example}
\newtheorem{remark}[lemma]{Remark}

\crefname{definition}{definition}{definitions}
\crefformat{definition}{#2definition~#1#3} 
\Crefformat{definition}{#2Definition~#1#3} 

\crefname{ex}{example}{examples}
\crefformat{example}{#2example~#1#3} 
\Crefformat{example}{#2Example~#1#3} 

\crefname{remark}{remark}{remarks}
\crefformat{remark}{#2remark~#1#3} 
\Crefformat{remark}{#2Remark~#1#3} 

\crefname{convention}{convention}{conventions}
\crefformat{convention}{#2convention~#1#3} 
\Crefformat{convention}{#2Convention~#1#3} 

\crefname{notation}{notation}{notations}
\crefformat{notation}{#2notation~#1#3} 
\Crefformat{notation}{#2Notation~#1#3} 

\crefname{table}{table}{tables}
\crefformat{table}{#2table~#1#3} 
\Crefformat{table}{#2Table~#1#3}

\crefname{lemma}{lemma}{lemmas}
\crefformat{lemma}{#2lemma~#1#3} 
\Crefformat{lemma}{#2Lemma~#1#3} 

\crefname{proposition}{proposition}{propositions}
\crefformat{proposition}{#2proposition~#1#3} 
\Crefformat{proposition}{#2Proposition~#1#3} 

\crefname{corollary}{corollary}{corollaries}
\crefformat{corollary}{#2corollary~#1#3} 
\Crefformat{corollary}{#2Corollary~#1#3} 

\crefname{theorem}{theorem}{theorems}
\crefformat{theorem}{#2theorem~#1#3} 
\Crefformat{theorem}{#2Theorem~#1#3} 

\crefname{enumi}{}{}
\crefformat{enumi}{(#2#1#3)}
\Crefformat{enumi}{(#2#1#3)}

\crefname{assumption}{assumption}{Assumptions}
\crefformat{assumption}{#2assumption~#1#3} 
\Crefformat{assumption}{#2Assumption~#1#3} 

\crefname{equation}{}{}
\crefformat{equation}{(#2#1#3)} 
\Crefformat{equation}{(#2#1#3)}

\numberwithin{equation}{section}

\theoremstyle{nonumberplain}
\theoremsymbol{\ensuremath{\blacksquare}}

\newtheorem{proof}{Proof}
\newcommand\pf[1]{\newtheorem{#1}{Proof of \Cref{#1}}}

\newcommand\bC{{\mathbb C}}

\newcommand\bG{{\mathbb G}}

\newcommand\bR{{\mathbb R}}
\newcommand\bS{{\mathbb S}}

\newcommand\cC{{\mathcal C}}
\newcommand\cD{{\mathcal D}}

\newcommand\cF{{\mathcal F}}

\newcommand\cH{{\mathcal H}}

\newcommand\cM{{\mathcal M}}

\newcommand\cV{{\mathcal V}}
\newcommand\cW{{\mathcal W}}

\DeclareMathOperator{\id}{id}

\newcommand{\cat}[1]{\textsc{#1}}

\newcommand\spr[1]{\cite[\href{https://stacks.math.columbia.edu/tag/#1}{Tag {#1}}]{stacks-project}}
\newcommand{\qedhere}{\mbox{}\hfill\ensuremath{\blacksquare}}

\title{Monadic forgetful functors and (non-)presentability for $C^*$- and $W^*$-algebras}
\author{Alexandru Chirvasitu and Joanna Ko}

\begin{document}

\date{}

\newcommand{\Addresses}{{
  \bigskip
  \footnotesize

  \textsc{Department of Mathematics, University at Buffalo, Buffalo,
    NY 14260-2900, USA}\par\nopagebreak \textit{E-mail address}:
  \texttt{achirvas@buffalo.edu}
  \quad \\
  
  \textsc{Department of Mathematics and Statistics, Masaryk University, 
  Kotlářská 2,
  Brno 61137, Czech Republic}\par\nopagebreak \textit{E-mail address}:
  \texttt{joanna.ko.maths@gmail.com}

}}

\maketitle

\begin{abstract}
  We prove that the forgetful functors from the categories of $C^*$- and $W^*$-algebras to Banach $*$-algebras, Banach algebras or Banach spaces are all monadic, answering a question of J.Rosick\'{y}, and that the categories of unital (commutative) $C^*$-algebras are not locally-isometry $\aleph_0$-generated either as plain or as metric-enriched categories, answering a question of I. Di Liberti and Rosick\'{y}.

  We also prove a number of negative presentability results for the category of von Neumann algebras: not only is that category not locally presentable, but in fact its only presentable objects are the two algebras of dimension $\le 1$. For the same reason, for a locally compact abelian group $\mathbb{G}$ the category of $\mathbb{G}$-graded von Neumann algebras is not locally presentable.
\end{abstract}


\vspace{.5cm}



\section*{Introduction}

Denote by $\cC^*_1$ and $\cC^*_{c,1}$ the categories of unital $C^*$- and unital commutative $C^*$-algebras respectively, and by $\cat{Ban}$ the category of Banach spaces with linear maps of norm $\le 1$ (i.e. {\it contractions}) as morphisms. The original impetus for the present paper was provided by \cite[Remark 5.2 (3)]{ros-mnd}, asking whether the forgetful functors
\begin{equation}\label{eq:gs}
  G:\cC^*_1\to \cat{Ban}\quad\text{and}\quad G_c:\cC^*_{c,1}\to \cat{Ban}
\end{equation}
are {\it monadic}. Settling for $\cC^*_1$ to fix ideas, recall ({\it tripleable} \cite[\S VI.3]{mcl} or \cite[\S 3.3]{bw}) that this means that $G$ induces an equivalence between $\cC^*_1$ and the {\it Eilenberg-Moore category} $\cat{Ban}^{GF}$ of algebras over the {\it monad} $FG$, where $F:\cat{Ban}\to \cC^*_1$ is the left adjoint of $G$.

Per \cite[\S VI.8]{mcl} or \cite[\S 3.31]{ar} the prototypical monadic functors are the forgetful functors $\cV\to \cat{Set}$ where $\cV$ is a {\it variety of algebras}: (possibly multi-sorted) sets equipped with operations required to satisfy various equations. Examples are the categories of semigroups, monoids, groups, rings, Lie algebras, modules over a fixed ring, etc. etc. The monadicity of the forgetful functors \Cref{eq:gs}, then, roughly says that the passage from Banach spaces to (unital) $C^*$-algebras is effected by freely adjoining a number of operations and equations.

We prove that the functors are indeed monadic in \Cref{cor:frgmnd}, along with a number of variations. Aggregating all of the claims of \Cref{th:cbanalg} and \Cref{cor:frgmnd}:

\begin{theoremN}
  The forgetful functors from the category $\cC^*_1$ of unital $C^*$-algebras to the categories of unital Banach $*$-algebras, unital Banach algebras and Banach spaces are all monadic.

  The same holds for commutative ($C^*$- and Banach) algebras.
\end{theoremN}

A number of other problems suggest themselves naturally, all circumscribed under the same general topic of studying operator algebras category-theoretically. For one thing, the obvious modification of the previous result goes through for {\it von Neumann} (or $W^*$-)algebras (see \Cref{th:wtocmnd} and \Cref{cor:wallmnd}, as well as the beginning of \Cref{se:wast} for a reminder on $W^*$-algebras):

\begin{theoremN}
  The forgetful functors from the category $\cW^*_1$ of $W^*$-algebras to the categories of $C^*$-algebras, unital Banach $*$-algebras, unital Banach algebras and Banach spaces are all monadic.
\end{theoremN}

In another direction, the tasks of proving monadicity, constructing (co)limit or adjoint functors, and so on simplify considerably when the categories in question are technically ``reasonable''. This might mean, perhaps, that every object can be recovered as a colimit of appropriately ``small'' objects. To recall, briefly, the relevant concepts from \cite[\S 1.B]{ar}:

\begin{definition}\label{def:lp}
  Let $\kappa$ be a regular cardinal, i.e. one that is not a union of strictly fewer strictly smaller sets.
  \begin{itemize}
  \item A poset $(I,\le)$ is {\it $\kappa$-directed} if every subset of cardinality $<\kappa$ has an upper bound.
  \item A {\it $\kappa$-directed colimit} in a category $\cC$ is a colimit of a functor $I\to \cC$ for a $\kappa$-directed poset $(I,\le)$, where the latter is regarded as a category with exactly one arrow $i\to j$ whenever $i\le j$.
  \item An object $c\in \cC$ is {\it $\kappa$-presentable} if $\mathrm{hom}_{\cC}(c,-):\cC\to \cat{Set}$ preserves $\kappa$-directed colimits. An object is {\it presentable} if it is $\kappa$-presentable for some $\kappa$.
  \item A category $\cC$ is {\it locally $\kappa$-presentable} if it is cocomplete and has a set $S$ of $\kappa$-presentable objects so that every object is a $\kappa$-directed colimit of objects from $S$. A category is {\it locally presentable} if it is locally $\kappa$-presentable for some $\kappa$.
  \end{itemize}  
\end{definition}

It has been known for some time that the categories $\cC^*_1$ and $\cC^*_{c,1}$ are locally $\aleph_1$-presentable (this follows, for instance, from \cite[Theorem 3.28]{ar} and the fact that they are varieties of infinitary algebras \cite[Theorem 2.4]{pr1}). It is natural to ask whether the category of $W^*$-algebras is also locally presentable:
\begin{itemize}
\item on the one hand it is certainly cocomplete (\Cref{pr:wccmpl}, \cite[Proposition 5.7]{krn-qcoll}, etc.);
\item while on the other hand the local-presentability claim does seem to occur in some of the literature (see \Cref{subse:npres} for examples).
\end{itemize}

We will nevertheless prove below a strong negation of local presentability for $\cW^*_1$ (\Cref{th:nopres} and \Cref{pr:grnpres}):

\begin{theoremN}
  The only presentable objects in the category $\cW^*_1$ of von Neumann algebras are $\{0\}$ and $\bC$.

  As a consequence, for any locally compact abelian group $\bG$, the category $\cW^*_{1,\bG}$ of $\bG$-graded von Neumann algebras in the sense of \Cref{def:grd} is not locally presentable.
\end{theoremN}

The notion of presentability (for an object) captures the intuition of being definable by ``few generators and relations'': in, say, categories of modules, being $\aleph_0$-presentable in the sense of \Cref{def:lp} is equivalent to being finitely-presentable in the usual sense (finitely many generators, finitely many relations) \cite[\S 3.10 (2)]{ar}. 

There is, similarly, an abstract formulation for the weaker notion of being definable by few generators (regardless of relations). Having fixed an appropriate class $\cM$ of monomorphisms in a category $\cC$ (half of a {\it factorization system} \cite[Definition 14.1]{ahs}; those details are not crucial here), recall, say, \cite[p.15, Definition]{ar-what} or \cite[Definitions 2.1 and 2.5]{dr}:

\begin{definition}\label{def:lg}
  Let $\cC$ be a cocomplete category and $\kappa$ a regular cardinal.
  \begin{itemize}
  \item An object $c\in \cC$ is {\it $\kappa$-generated with respect to (wrt) $\cM$} if $\mathrm{hom}_{\cC}(c,-)$ preserves $\kappa$-directed colimits of morphisms from $\cM$.
  \item $\cC$ is {\it $\cM$-locally $\kappa$-generated} if there is a set $S$ of $\kappa$-generated objects wrt $\cM$ such that every object is expressible as a $\kappa$-directed colimit of $\cM$-morphisms between objects from $S$.
  \item $\cC$ is {\it $\cM$-locally generated} if it is $\cM$-locally $\kappa$-generated for some regular cardinal $\kappa$.
  \end{itemize}
\end{definition}

When $\cM$ is the class of all monomorphisms this recovers \cite[Definition 1.67]{ar}, and as for presentability, the concept embodies the right intuition in familiar cases: an object in a category of modules is $\aleph_0$-generated wrt the class of monomorphisms if and only if it is finitely-generated in the usual sense \cite[\S 3.10 (1)]{ar}.

\cite[Remark 6.10]{dr} speculates that the category $\cC^*_1$ is not isometry-locally $\aleph_0$-generated. The $\aleph_0$-generation notion discussed there is not {\it quite} that of \Cref{def:lg}, but rather an enriched version thereof \cite[Definitions 4.1 and 4.4]{dr}: everything in sight (Banach spaces/algebras, $C^*$-algebras) is regarded as enriched over the category $\cat{CMet}$ of complete generalized metric spaces, where `generalized' means the distance is allowed values in $[0,\infty]$. One can then reiterate \Cref{def:lg} by requiring that the colimit-preservation condition
\begin{equation*}
  \mathrm{hom}(c,\ \varinjlim_i c_i)\cong \varinjlim_i \mathrm{hom}(c,c_i)
\end{equation*}
take place in the enriching category $\cat{CMet}$ instead of $\cat{Set}$. Either way (enriched or not), we can confirm that speculation: see \Cref{pr:fdim}, \Cref{cor:ngen}, \Cref{pr:enrcfdim} and \Cref{cor:enrcngen}.

\begin{theoremN}
  Let $A$ be a commutative unital $C^*$-algebra and $\cM$ the class of isometric $C^*$ morphisms.
  \begin{enumerate}[(a)]
  \item $A$ is $\aleph_0$-generated wrt $\cM$ in the (plain or enriched) category $\cC^*_{c,1}$ if and only if it is finite-dimensional.
  \item $A$ is $\aleph_0$-generated wrt $\cM$ in the ordinary category $\cC^*_1$ if and only if it has dimension $\le 1$.
  \item $A$ is $\aleph_0$-generated wrt $\cM$ in the $\cat{CMet}$-enriched category $\cC^*_1$ if and only if it is finite-dimensional.
  \end{enumerate}
  Consequently, $\cC^*_1$ and $\cC^*_{c,1}$ are not isometry-locally $\aleph_0$-generated, either as plain categories or as $\cat{CMet}$-enriched categories.
\end{theoremN}

\subsection*{Acknowledgements}

The first author was partially funded by NSF grant DMS-2001128.

We are grateful to J.Rosick\'{y} for many insightful comments and questions.

\section{Preliminaries}\label{se:prel}

The requisite category-theoretic background is amply covered in, say, \cite{mcl,bw,ar}, with more precise references given below, as needed.

For objects $x,y$ of a category $\cC$, both $\mathrm{hom}_{\cC}(x,y)$ and $\cC(x,y)$ denote the respective set of morphisms, and we depict adjunctions $(F,G)$ with $F$ as the left adjoint as $F\dashv G$.

We also assume some material on operator algebras ($C^*$ or $W^*$), for which the reader can consult any number of excellent sources (some cited below in more detail): \cite{blk,tak1,dixc,dixw}, etc.

We write
\begin{itemize}
\item $\cC^*_1$ and $\cC^*_{c,1}$ for the categories of unital (commutative) $C^*$-algebras respectively;
\item $\cat{BanAlg}_1$ and $\cat{BanAlg}_{c,1}$ for the categories of unital (commutative) Banach algebras respectively;
\item $\cat{BanAlg}^*_1$ and $\cat{BanAlg}^*_{c,1}$ for the categories of unital (respectively commutative) Banach $*$-algebras;
\item and finally, $\cat{Ban}$ for the category of complex Banach spaces with contractions (these being the ``appropriate'' morphisms when handling Banach spaces category-theoretically \cite[\S 1.48]{ar}).
\end{itemize}

\section{Monadic categories of $C^*$-algebras}\label{se:cast}

\cite[Remark 5.2 (3)]{ros-mnd} asks whether the forgetful functors
\begin{equation*}
  G:\cC^*_1\to \cat{Ban}\quad\text{and}\quad G_c:\cC^*_{c,1}\to \cat{Ban}
\end{equation*}
are {\it monadic} \cite[\S VI.3]{mcl} (or {\it tripleable} \cite[\S 3.3]{bw}), i.e. whether they can be identified with the forgetful functors from the categories of algebras for the monads $GF$ and $G_c F_c$ attached to the adjunctions
\begin{equation*}
  F\dashv G\quad\text{and}\quad F_c\dashv G_c.
\end{equation*}
\Cref{cor:frgmnd} gives affirmative answers to those two questions.

We first recall some category-theoretic language and background. Recall (\cite[\S 3.3]{bw}):

\begin{definition}\label{def:spl}
  For a category $\cC$:
  \begin{itemize}
  \item A pair of morphisms $\partial_i:A\to B$ $i=0,1$ in $\cC$ is {\it reflexive} if the two arrows have a common {\it section}, i.e. a morphism $t:B\to A$ with
    \begin{equation*}
      \partial_1 t = \id_B = \partial_0 t.
    \end{equation*}
  \item A {\it reflexive coequalizer} is a coequalizer of a reflexive pair.
  \item A pair $\partial_i: A\to B$ is {\it contractible} or {\it split} if there is an arrow $t:B\to A$ such that
    \begin{equation*}
      \partial_0 t = \id_B,\quad \partial_1\circ t\circ\partial_1 = \partial_1 \circ t\circ \partial_0. 
    \end{equation*}
  \item A {\it contractible} or {\it split coequalizer} in $\cC$ consists of a diagram
    \begin{equation}\label{eq:spl}
      \begin{tikzpicture}[auto,baseline=(current  bounding  box.center)]
        \path[anchor=base] 
        (0,0) node (l) {$A$}
        +(4,0) node (m) {$B$}
        +(8,0) node (r) {$C$}
        ;
        \draw[->] (l) to[bend left=26] node[pos=.5,auto] {$\scriptstyle \partial_0$} (m);
        \draw[<-] (l) to[bend left=0] node[pos=.5,auto] {$\scriptstyle t$} (m);
        \draw[->] (l) to[bend left=-26] node[pos=.5,auto,swap] {$\scriptstyle \partial_1$} (m);
        \draw[->] (m) to[bend left=16] node[pos=.5,auto] {$\scriptstyle e$} (r);
        \draw[<-] (m) to[bend left=-16] node[pos=.5,auto,swap] {$\scriptstyle s$} (r);
      \end{tikzpicture}
    \end{equation}
    such that
    \begin{equation*}
      e\partial_0 = e\partial_1,\quad es=\id,\quad \partial_0 t=\id,\quad \partial_1t = se.
    \end{equation*}
    Just the presence of the arrows and the equations automatically implies that $e$ is a coequalizer for $\partial_i$, $i=0,1$ (\cite[\S 3.3, Proposition 2 (a)]{bw} or \cite[\S VI.6, Lemma]{mcl}), hence the name (`split {\it coequalizer}'). Furthermore, if a contractible pair of parallel arrows has a coequalizer, it is split \cite[\S 3.3, Proposition 2 (c)]{bw}; this justifies the terminology coincidence in the last two bullet items.
  \end{itemize}
  Given a functor $G:\cC\to \cD$, a pair $\partial_i:A\to B$, $i=0,1$ in $\cC$ is {\it $G$-split or $G$-contractible} if the pair $G\partial_i$ is contractible in $\cD$. The notion of being {\it $G$-reflexive} is defined analogously.
\end{definition}

Various criteria ensure that functors are monadic, each useful under appropriate circumstances. Recall two such sets of criteria (\cite[\S 3.3 Theorem 10 and subsequent discussion]{bw} as well as \cite[\S 3.5, paragraph preceding Proposition 1]{bw}):

\begin{definition}\label{def:tt}
  A functor $G:\cC\to \cD$
  \begin{itemize}
  \item satisfies the {\it Crude Tripleability Theorem (CTT)} if
  \begin{enumerate}[(1)]
  \item it has a left adjoint;
  \item it reflects isomorphisms;
  \item $\cC$ has coequalizers for those reflexive pairs $\partial_i$, $i=0,1$ for which $G\partial_i$ has a coequalizer, and $G$ preserves those coequalizers. 
  \end{enumerate}
\item satisfies the {\it Precise Tripleability Theorem (PTT)} if
  \begin{enumerate}[(1)]
  \item it has a left adjoint;
  \item it reflects isomorphisms;
  \item $\cC$ has coequalizers for reflexive $G$-split pairs, and $G$ preserves them.
  \end{enumerate}
  \end{itemize}
  We also use `CTT' and `PTT' as adjectives (e.g. `the functor $G$ is PTT').

  According to {\it Beck's theorem} (\cite[\S 3.3, Theorem 10]{bw}), being PTT is equivalent to being monadic, so these terms (along with `tripleable' will be interchangeable). On the other hand, as noted in \cite[\S 3.5, paragraph preceding Proposition 1]{bw}, being CTT is {\it sufficient} for monadicity.
\end{definition}

The following remark is a consequence of \cite[\S 3.5, Proposition 1 (b)]{bw}; we include a proof for completeness.

\begin{lemma}\label{le:cttptt}
  Consider functors $G'' \colon \mathscr{A} \to \mathscr{B}$ and $G' \colon \mathscr{B} \to \mathscr{C}$. If $G''$ is CTT and $G'$ is monadic then the composition $G'\circ G''$ is monadic.
\end{lemma}
\begin{proof}
  Since right adjoints and conservative functors are closed under composition, it suffices to show $\mathscr{A}$ has coequalizers of reflexive $G'G''$\nobreakdash-split coequalizer pairs, and $G'G''$ preserves them.
  
  Suppose we have a reflexive $G'G''$\nobreakdash-split coequalizer pair in $\mathscr{A}$:
  \begin{equation}\label{eq:reflexive}
    \begin{tikzpicture}[auto,baseline=(current  bounding  box.center)]
      \path[anchor=base] 
      (0,0) node (l) {$A$}
      +(4,0) node (m) {$B$}
      ;
      \draw[->] (l) to[bend left=26] node[pos=.5,auto] {$\scriptstyle 
        \partial_0$} (m);
      \draw[<-] (l) to[bend left=0] node[pos=.5,auto] {$\scriptstyle t$} (m);
      
      \draw[->] (l) to[bend left=-26] node[pos=.5,auto,swap] {$\scriptstyle 
        \partial_1$} (m);
    \end{tikzpicture}
  \end{equation}
  such that $t$ is the common section for $\partial_0$ and $\partial_1$, for which there is a split coequalizer diagram in $\mathscr{C}$:
  \begin{equation}
    \begin{tikzpicture}[auto,baseline=(current  bounding  box.center)]
      \path[anchor=base] 
      (0,0) node (l) {$G'G''A$}
      +(4,0) node (m) {$G'G''B$}
      +(8,0) node (r) {$Z_{\mathscr{C}}$}
      ;
      \draw[->] (l) to[bend left=26] node[pos=.5,auto] {$\scriptstyle 
        G'G''\partial_0$} (m);
      \draw[<-] (l) to[bend left=0] node[pos=.5,auto] {$\scriptstyle b$} (m);
      
      \draw[->] (l) to[bend left=-26] node[pos=.5,auto,swap] 
      {$\scriptstyle 
        G'G''\partial_1$} (m);
      \draw[->>] (m) to[bend left=16] node[pos=.5,auto] {$\scriptstyle 
        e_{\mathscr{C}}$} (r);
      \draw[<-] (m) to[bend left=-16] node[pos=.5,auto,swap] {$\scriptstyle 
        s$} (r);
    \end{tikzpicture}.
  \end{equation}
  Since $\mathscr{A}$ has reflexive coequalizers, there is a coequalizer $e \colon B \twoheadrightarrow Z$ for \Cref{eq:reflexive}. It suffices to show $G'G''$ preserves the coequalizer.

  Since $G''$ preserves reflexive coequalizers, $G''e \colon G''B \twoheadrightarrow G''Z$ is the coequalizer for the image of \Cref{eq:reflexive} under $G''$ in $\mathscr{B}$. Since $\mathscr{B}$ has coequalizers of reflexive $G'$\nobreakdash-split coequalizer pairs, there is a coequalizer diagram in $\mathscr{B}$:
  \begin{equation}
    \begin{tikzpicture}[auto,baseline=(current  bounding  box.center)]
      \path[anchor=base] 
      (0,0) node (l) {$G''A$}
      +(4,0) node (m) {$G''B$}
      +(8,0) node (r) {$Z_{\mathscr{B}}$}
      ;
      \draw[->] (l) to[bend left=26] node[pos=.5,auto] {$\scriptstyle 
        G''\partial_0$} (m);
      \draw[<-] (l) to[bend left=0] node[pos=.5,auto] {$\scriptstyle G''t$} 
      (m);
      
      \draw[->] (l) to[bend left=-26] node[pos=.5,auto,swap] 
      {$\scriptstyle 
        G''\partial_1$} (m);
      \draw[->>] (m) to[bend left=0] node[pos=.5,auto] {$\scriptstyle 
        e_{\mathscr{B}}$} (r);
    \end{tikzpicture}
  \end{equation}
  and since $G'$ preserves them, $G'e_{\mathscr{B}} \cong e_{\mathscr{C}}$. Now by the uniqueness of coequalizers, we have $G''e \cong e_{\mathscr{B}}$ and so $G'G''e \cong e_{\mathscr{C}}$.
\end{proof}

\begin{theorem}\label{th:cbanalg}
  The forgetful functors
  \begin{equation*}
    G:\cC^*_1\to \cat{BanAlg}^*_1\quad\text{and}\quad G_c:\cC^*_{c,1}\to \cat{BanAlg}^*_{c,1}
  \end{equation*}
  are both CTT and hence monadic.
\end{theorem}
\begin{proof}
  We focus on $\cC^*_1$, as the other argument is entirely parallel.

  That $G$ is a right adjoint we can see as in, say, \cite[\S 5]{ros-mnd} (which discusses the forgetful functor to $\cat{Ban}$ instead). Isomorphism-reflection, on the other hand, follows from the fact that bijective morphisms of $C^*$-algebras are automatically isometries (and hence invertible) \cite[discussion preceding Theorem 1.3.2]{arv}.
  
  The existence of coequalizers (reflexive or not) is not an issue: not only is $\cC^*_1$ cocomplete (i.e. has arbitrary colimits), but as noted in \cite[Remark 6.10]{dr}, it is {\it locally $\aleph_1$-presentable} in the sense of \cite[Definition 1.17]{ar}, because it can be realized \cite[Theorem 2.4]{pr1} as a variety of algebras equipped with $\aleph_0$-ary operations \cite[Theorem 3.28]{ar}.
  
  It thus remains to argue that $G$ preserves reflexive coequalizers. Here too, much more is true: it {\it creates} {\it arbitrary} coequalizers (\Cref{le:crcoeq}).
\end{proof}

\begin{lemma}\label{le:crcoeq}
  The forgetful functor
  \begin{equation*}
    G:\cC^*_1\to \cat{BanAlg}^*_1
  \end{equation*}
  from unital $C^*$-algebras to unital Banach $*$-algebras creates arbitrary coequalizers in the sense of \cite[\S 1.7, discussion following Proposition 3]{bw}.

  The analogous statement holds for the respective categories of commutative $C^*$- and Banach $*$-algebras.
\end{lemma}
\begin{proof}
  We once more focus on the non-commutative version to fix ideas, but the argument goes through virtually verbatim in general.

  Consider a parallel pair $f,g:A\to B$ of unital $C^*$-morphisms. The coequalizer in $\cat{BanAlg}^*_1$ is obtained by annihilating the closed ideal
  \begin{equation*}
    I:=\overline{B\cdot \{f(a)-g(a)\ |\ a\in A\}\cdot B}
  \end{equation*}
  and then equipping the quotient $B/I$ with the largest Banach-space norm making $\pi:B\to B/I$ contractive: for $x\in B/I$,
  \begin{equation*}
    \|x\| := \inf \{\|b\|\ |\ b\in B,\ \pi(b)=x\}.
  \end{equation*}
  But because $I\trianglelefteq B$ is a closed $*$-ideal, this is already a $C^*$-norm on $B/I$ \cite[\S 1.3, Corollary 2]{arv}. In conclusion, for two parallel unital $C^*$-morphisms the coequalizer constructions in $\cC^*$ and $\cat{BanAlg}^*_1$ coincide.
\end{proof}

\Cref{th:cbanalg} in turn implies 

\begin{corollary}\label{cor:frgmnd}
  The following forgetful functors are all monadic:
  \begin{enumerate}[(a)]
  \item From $\cC^*_1$ to unital Banach $*$-algebras, unital Banach algebras, or Banach spaces.
  \item From $\cC^*_{c,1}$ to unital commutative Banach $*$-algebras, unital commutative Banach algebras, or Banach spaces.
  \end{enumerate}
\end{corollary}
\begin{proof}    
  In each case the functor in question decomposes as one of the CTT (\Cref{th:cbanalg}) forgetful functors
  \begin{equation*}
    \cC^*_1\to \cat{BanAlg}^*_1\quad\text{or}\quad\cC^*_{c,1}\to \cat{BanAlg}^*_{c,1},
  \end{equation*}
  followed by forgetful functors
  \begin{equation*}
    \cat{BanAlg}^*_1\to \cat{BanAlg}_1,\quad \cat{BanAlg}^*_1\to \cat{Ban}
  \end{equation*}
  or 
  \begin{equation*}
    \cat{BanAlg}^*_{c,1}\to \cat{BanAlg}_{c,1},\quad \cat{BanAlg}^*_{c,1}\to \cat{Ban}.
  \end{equation*}
  That these last four functors are monadic follows either directly from \cite[Theorem 5.1]{ros-mnd} or from very slight alterations to its proof, so the conclusion is a consequence of \Cref{le:cttptt}.
\end{proof}


\section{$\aleph_0$-generation}

As noted in \cite[Remark 6.10]{dr} (and recalled above in the course of the proof of \Cref{cor:frgmnd}) the categories $\cC^*_1$ and $\cC^*_{c,1}$ are both locally $\aleph_1$-presentable.

By contrast, the same \cite[Remark 6.10]{dr} speculates that $\cC^*_1$ is unlikely to be $\cM$-locally $\aleph_0$-generated, where $\cM$ is the class of unital $C^*$-embeddings (which are, in particular, automatically isometric).  That this is indeed the case will follow from the identification of those commutative $C^*$-algebras that are $\aleph_0$-generated with respect to $\cM$.  As much of \cite[Remark 6.10]{dr} makes sense for $\cC^*_{1}$ and $\cC^*_{c,1}$ either as plain categories or as categories enriched over metric spaces, we treat discuss the settings separately.

\subsection{Ordinary categories of $C^*$-algebras}\label{subse:plain}

\begin{proposition}\label{pr:fdim}
  A commutative $C^*$-algebra is $\aleph_0$-generated w.r.t. the class $\cM$ of $C^*$ embeddings
  \begin{enumerate}[(a)]
  \item\label{item:1} in $\cC^*_{c,1}$ if and only if it is finite-dimensional;
  \item\label{item:2} and in $\cC^*_1$ if and only if it has dimension $\le 1$.
  \end{enumerate}
\end{proposition}
\begin{proof}
  Some claims can be treated simultaneously for both \Cref{item:1} and 
  \Cref{item:2}.

  {\bf ($\Leftarrow$): algebras of dimension $\le 1$.} That is, either the zero algebra of the scalars. That these are indeed $\cM$-generated (in either category) is immediate.

  {\bf ($\Rightarrow$): ruling out infinite-dimensional algebras in both \Cref{item:1} and \Cref{item:2}.} Consider an $\cM$-$\aleph_0$-generated commutative $C^*$-algebra $A\in \cC^*_{c,1}$. We have $A\cong C(X)$ for some compact $X$ by Gelfand-Naimark \cite[Theorem III.2.2.4]{blk}. Denoting by $X_d$ the underlying set $X$ with the discrete topology, we now have a surjection
  \begin{equation*}
    \beta X_d\to X
  \end{equation*}
  from the {\it Stone-\v{C}ech compactification} \cite[\S 38, Definition preceding Exercises]{mnk} of $X_d$. Because $X_d$ is discrete its Stone-\v{C}ech compactification is totally disconnected \cite[\S 38, Exercise 7 (c)]{mnk}, and hence {\it profinite} \spr{08ZY}: it is a filtered limit of finite (discrete) spaces, say
  \begin{equation*}
    \beta X_d = \varprojlim_i X_i,\quad |X_i|<\infty.
  \end{equation*}
  Dualizing back to commutative $C^*$-algebras, this gives us a one-to-one morphism
  \begin{equation*}
    C(X)\to \varinjlim_i C(X_i)
  \end{equation*}
  into a directed colimit of embeddings. The $\aleph_0$-generation hypothesis then ensures that said embedding factors through some $C(X)\to C(X_i)$, and since $C(X_i)$ is finite-dimensional $A\cong C(X)$ must be too. 

  {\bf \Cref{item:2} ($\Rightarrow$): ruling out algebras of dimension $\ge 2$ in $\cC^*_1$.} First, note that the isometry-$\aleph_0$-generation property is inherited by quotients, so it suffices to focus on the two-dimensional $C^*$-algebra $A\cong \bC^2$.

  A unital morphism $A\to B$ simply picks out a projection in $B$. All the statement claims, then, is that there are directed colimits
  \begin{equation*}
    B=\varinjlim_i B_i
  \end{equation*}
  of unital $C^*$-algebras that contain projections which belong to none of the individual $B_i$. This is well known; \cite[Example 1.3]{laz-dir}, for instance, explains how to construct such a projection (denoted on \cite[p.711]{laz-dir} by $x$) in the ``{\it Fermion algebra}'' of \cite[discussion preceding Proposition 6.4.3]{ped-aut}: the directed limit
  \begin{equation*}
    B=\varinjlim_n M_{2^n}
  \end{equation*}
  of the inclusions
  \begin{equation*}
    M_{2^n}\ni a\mapsto a\otimes 1\in M_{2^n}\otimes M_2\cong M_{2^{n+1}}. 
  \end{equation*}
  

  {\bf \Cref{item:1} ($\Leftarrow$): arbitrary finite-dimensional algebras in $\cC^*_{c,1}$.} The finite-dimensional commutative $C^*$-algebras are those of continuous functions on finite discrete spaces, so upon dualizing by Gelfand-Naimark the claim is as follows: for every filtered limit
  \begin{equation*}
    X=\varprojlim_i X_i
  \end{equation*}
  of (surjections of) compact Hausdorff spaces, a continuous map $\pi:X\to F$ to a finite discrete set must factor through one of the surjections $\pi_i:X\to X_i$.

  For a map $f$ defined on $X$ (such as $\pi$ or the $\pi_i$) set
  \begin{equation*}
    \cat{diff}(f):=\{(x,y)\in X^2\ |\ f(x)\ne f(y) \}.
  \end{equation*}
  These sets are always open for continuous $f$ into Hausdorff spaces.

  The realization of $X$ as a limit ensures that every $(x,y)\in \cat{diff}(\pi)$ has some open neighborhood of the form
  \begin{equation*}
    U_{x,y,i}\subseteq \cat{diff}(\pi_i)
  \end{equation*}
  for some $i$. Because $\pi:X\to F$ has finite discrete codomain $\cat{diff}(\pi)$ is also closed in $X^2$, and hence compact covered by only finitely many $U_{x,y,i}$. The map $\pi$ will then factor through $\pi_j$ for $j$ dominating all $i$ appearing among these finitely many $U_{x,y,i}$, finishing the proof.
\end{proof}

As announced above, \Cref{pr:fdim} implies

\begin{corollary}\label{cor:ngen}
  The categories $\cC^*_1$ and $\cC^*_{c,1}$ are not isometry-locally $\aleph_0$-generated.
\end{corollary}
\begin{proof}
  Indeed, according to \Cref{pr:fdim} the only commutative unital $C^*$-algebras that can be recovered as directed colimits of embeddings of isometry-$\aleph_0$-generated subalgebras are the {\it AF algebras} of \cite[Definition 7.1.1]{blk-k}.

  Dualizing via Gelfand-Naimark \cite[Theorem III.2.2.4]{blk}, these are the algebras of the form $C(X)$ for profinite $X$. For that reason, commutative $C^*$-algebras cannot all be such colimits. 
\end{proof}

\subsection{Metric-space-enriched categories}\label{subse:enrc}

\cite[\S 1.2]{kly} defines the notion of a {\it $\cV$-category} (often also called a {\it $\cV$-enriched category}) for a monoidal \cite[\S 1.1]{kly} category $\cV$. That source quickly specializes in \cite[\S 1.6]{kly} to {\it symmetric monoidal closed} \cite[\S\S 1.4, 1.5]{kly} categories $\cV$. 

\cite[\S 6]{dr} places $\cC^*_{1}$ and $\cC^*_{c,1}$ (and categories of Banach spaces, etc.) in this enriched context, with $\cat{CMet}$ \cite[Examples 2.3 (2)]{ar-ap} as the enriching category $\cV$. This is the category of complete {\it generalized} metric spaces with {\it non-expansive maps} (also: `contractions') as morphisms:
\begin{equation*}
  f:(X,d_X)\to (Y,d_Y),\quad d_Y(fx,fx')\le d_X(x,x'),\ \forall x,x'\in X,
\end{equation*}
and `generalized' means that the distance is allowed infinite values. We have an enriched counterpart to \Cref{pr:fdim}.

\begin{proposition}\label{pr:enrcfdim}
  A commutative $C^*$-algebra is isometry-locally $\aleph_0$-generated in either of the $\cat{CMet}$-enriched categories $\cC^*_1$ or $\cC^*_{c,1}$ if and only if it is finite-dimensional. 
\end{proposition}
\begin{proof}
  We prove the two implications separately.

  {\bf (1): finite-dimensional commutative $\Rightarrow$ $\aleph_0$-generated.} Since finite-dimensional commutative $C^*$-algebras are of the form $\bC^n$, we have to examine morphisms
  \begin{equation*}
    \phi:\bC^n\to A=\varinjlim_i A_i
  \end{equation*}
  for a colimit of embeddings. The projection $e_1:=(1,0,\cdots,0)\in \bC^n$ is arbitrarily approximable by a projection $p\in A_i$ for some $i$ \cite[Proposition L.2.2]{wo}, so our original morphism $\phi:\bC^n\to A$ is arbitrarily approximable by morphisms mapping $e_1$ to such a projection $p\in A_i$.

  Choosing such an approximation, assume now that $\phi(e_1) = p\in A_i$. We can work only with the indices $j\ge i$ (so that $p$ can be regarded as a projection in {\it all} $A_j$), and consider the morphism
  \begin{equation*}
    \phi|_{(1-e_1)\bC^n}:(1-e_1)\bC^n\to pAp = \varinjlim_i pA_ip.
  \end{equation*}
  A repetition of the previous procedure will now allow us to detach a further minimal projection $e_2$ from the $(n-1)$-dimensional $(1-e_1)\bC^n\cong \bC^{n-1}$ and assume, upon approximating, that $\phi(e_2)=q\in A_j$ for some $j$, etc.
  
  In short: the original morphism $\phi:\bC^n\to A$ is arbitrarily approximable, uniformly on the unit ball of $\bC^n$, by $C^*$-morphisms that factor through some $A_i$.
  
  {\bf (2): $\aleph_0$-generated commutative $\Rightarrow$ finite-dimensional.} We can repurpose the relevant portion of the proof of \Cref{pr:fdim}: write, once more, $A=C(X)$ for an infinite compact Hausdorff $X$, and consider the embedding
  \begin{equation*}
    \iota:C(X)\to C(\beta X_d)
  \end{equation*}
  for the Stone-\v{C}ech compactification $\beta X_d$ of the discrete space underlying $X$.

  $\beta X_d$ is the inverse limit of its surjections $\beta X_d\to X_\cF$ onto the finite discrete space associated to a finite partition $\cF$ of $X_d$, so that
  \begin{equation*}
    C(\beta X_d)\cong \varinjlim_{\cF}C(X_\cF). 
  \end{equation*}
  If $C(X)$ were isometry-locally $\aleph_0$-generated in the enriched sense, we could at least approximate $\iota$ arbitrarily well by morphisms $C(X)\to C(X_\cF)$ for various finite partitions $\cF$ of $X$. This means that for every $\varepsilon>0$ there is some finite partition $\cF=\cF_{\varepsilon}$ such that 
  \begin{equation*}
    \forall f\in C(X)\ \exists f'\in C(X_{\cF}):\ \|f-f'\|<\varepsilon\|f\|. 
  \end{equation*}
  In other words: every continuous function $f\in C(X)$ is $\varepsilon\|f\|$-close to some function $f'$ constant along every part of the finite partition
  \begin{equation*}
    \cF=(F_0,\ \cdots,\ F_n),\quad X=\bigsqcup_i F_i
  \end{equation*}
  But if $X$ is infinite, some $F_i$ contains two distinct points $x\ne y\in X$. We can then find a continuous
  \begin{equation*}
    f\in C(X),\ \|f\|=1,\ f(x)=0,\ f(y)=1,
  \end{equation*}
  and such a function is at least $1$ apart in the uniform norm from any $f'\in C(X_{\cF})$.
\end{proof}

This in turn entails the enriched version of \Cref{cor:ngen}.

\begin{corollary}\label{cor:enrcngen}
  The $\cat{CMet}$-enriched categories $\cC^*_1$ and $\cC^*_{c,1}$ are not isometry-locally $\aleph_0$-generated.
\end{corollary}
\begin{proof}
  As in the proof of \Cref{cor:ngen}, using \Cref{pr:enrcfdim} in place of \Cref{pr:fdim}.
\end{proof}

With slightly more effort, we can transport part of \Cref{pr:enrcfdim} over to non-commutative $C^*$-algebras. The key point is

\begin{proposition}\label{pr:apmn}
  For any finite-dimensional $C^*$-algebra $B$, a unital $C^*$-morphism
  \begin{equation*}
    \phi:B\to A:=\varinjlim_i A_i
  \end{equation*}
  into a directed colimit of embeddings is arbitrarily approximable by morphisms $B\to A_i$.
\end{proposition}
\begin{proof}
  Finite-dimensional $C^*$-algebras are finite products
  \begin{equation*}
    B\cong M_{n_1}\times\cdots\times M_{n_k}.
  \end{equation*}
  Separating out the individual central factors $M_{n_i}$ by means of minimal central projections as in the proof of \Cref{pr:enrcfdim}, we can focus on the individual matrix algebras $M_{n_i}$. In short, this allows us to assume $B=M_n$ for the rest of the proof.

  The matrix algebra $M_n$ can be realized as the universal $C^*$-algebra generated by two unitaries $S$ and $U$, with
  \begin{equation}\label{eq:usrels}
    S^n=U^n=1,\quad USU^{-1}=\zeta S
  \end{equation}
  for some primitive $n^{th}$ root of unity $\zeta$.

  We will henceforth identify $M_n$ with its image through (the automatically-injective) $\phi:M_n\to A$. The argument consists of a number of steps. Throughout, we write `$\simeq$' to mean `is close to'. This will avoid having to keep careful track of $\varepsilon$ estimates, 
  \begin{enumerate}[(1)]
  \item First, $U$ can be approximated arbitrarily well with some order-$n$ unitary $U'\in A_i$, since the $C^*$-algebra it generates is commutative and finite-dimensional, and the commutative case has already been handled (in \Cref{pr:enrcfdim}). Working only with indices dominating $i$, we can henceforth assume that all $A_j$ contain $U'$.
  \item Next, $S$ was a $\zeta$-eigenvector for $U$ and $U'$ is close to $U$, so the $U'$-$\zeta$-eigenvector
    \begin{equation*}
      S':=\frac 1n\sum_{s=0}^{n-1} \zeta^{-s}\cdot (U')^s S (U')^{-s}
    \end{equation*}
    is close to $S$, because
    \begin{equation*}
      S=\frac 1n\sum_{s=0}^{n-1} \zeta^{-s}\cdot U^s S U^{-s}.
    \end{equation*}
  \item Since conjugation by $U'$ leaves all $A_i$ invariant, we can further assume that $S'\simeq S$ belongs to one (and hence all) of the $A_i$.
  \item Being close to $S$ the, element $S'$ is invertible \cite[Lemma 4.2.1]{wo}. Its {\it polar decomposition} \cite[\S 1.5]{wo}
    \begin{equation*}
      S'=S'' P,\ S'', P\in A_i
    \end{equation*}
    then provides a {\it unitary} $U'$-$\zeta$-eigenvector $S''$. We still have $S''\simeq S$, because $S$ itself was unitary to begin with.
  \item The continuity of the spectrum on normal operators \cite[Solution 105]{hlm-hs} together with $S''\simeq S$ and
    \begin{equation*}
      U' S'' (U')^{-1} = \zeta S''
    \end{equation*}
    show that the spectrum $\sigma(S'')$ of $S''$ is a subset of the unit circle invariant under multiplication by $\zeta$ and concentrated around $\{\zeta^s\ |\ 0\le s<n\}$. This latter condition means that $\sigma(S'')$ is contained in the union of $n$ mutually-disjoint closed arcs
    \begin{equation*}
      I_s\ni \zeta^s,\quad 0\le s<n.
    \end{equation*}
    We can then find a continuous self-map $f:\bS^1\to \bS^1$, equivariant under multiplication by $\zeta$, that compresses each $I_s$ onto $\zeta^s$. The operator $f(S'')$ obtained by {\it functional calculus} \cite[\S II.2.3]{blk} now has all of the desired properties:
    \begin{itemize}
    \item it is unitary of order $n$;
    \item satisfies
      \begin{equation*}
        U' f(S'') (U')^{-1} = \zeta f(S'');
      \end{equation*}
    \item and is contained in $A_i$ and close to $S''\simeq S$.
    \end{itemize}
  \end{enumerate}
  We have thus achieved the desired goal of approximating $U$ and $S$ by unitaries in $A_i$ satisfying the same defining relations \Cref{eq:usrels} (with $U'$ and $f(S'')$ in place of $U$ and $S$ respectively). 
\end{proof}

The announced consequence is now simply a rephrasing of \Cref{pr:apmn}. 

\begin{corollary}\label{cor:apmn}
  Finite-dimensional $C^*$-algebras are isometry-locally $\aleph_0$-generated in the $\cat{CMet}$-enriched category $\cC^*_1$.  \qedhere
\end{corollary}

The adjunctions between
\begin{itemize}
\item $\cC^*_1$ and $\cat{Ban}$;
\item and similarly, $\cC^*_{c,1}$ and $\cat{Ban}$
\end{itemize}
can be put to much good use, but we caution the reader that they are not enriched over $\cat{CMet}$. This is already visible when working with the simplest examples, as the following discussion prompted by \cite{ros-email} illustrates (focusing on $\cC^*_{c,1}$, i.e. {\it commutative} $C^*$-algebras).

\begin{example}\label{ex:adjnotenrc}
  As noted on \cite[p.173]{sem-bm}, the left adjoint $F$ to the forgetful functor $\cC^*_{c,1}\to\cat{Ban}$ (the so-called {\it Banach-Mazur functor}) can be described explicitly as
  \begin{equation*}
    \cat{Ban}\ni X\mapsto C(X^*_1)\in \cC^*_{c,1},
  \end{equation*}
  where $X^*_1$ is the unit ball of the Banach-space dual $X^*$, equipped with the weak$^*$ topology (wherein $X^*_1$ is compact, by the Banach-Alaoglu theorem \cite[Theorem 3.15]{rud-fa}).

  Consider the one-dimensional Banach space $\bC$. We then have $F(\bC)\cong C(D)$, where $D\subset \bC$ is the unit disk. The adjunction $F \dashv \cat{forget}$ gives a bijection
  \begin{equation*}
    \mathrm{hom}_{\cC^*_{c,1}}(C(D),\bC)\cong \mathrm{hom}_{\cat{Ban}}(\bC,\bC),
  \end{equation*}
  which however is not an isometry in $\cat{CMet}$ (or indeed, even a homeomorphism). To see this, note that
  \begin{itemize}
  \item The metric space $\mathrm{hom}_{\cat{Ban}}(\bC,\bC)$ is $D$ with its usual metric, identifying a contraction
    \begin{equation*}
      T:\bC\to \bC
    \end{equation*}
    with $T(1)\in D$.
  \item On the other hand, $\mathrm{hom}_{\cC^*_{c,1}}(C(D),\bC)$ is the {\it spectrum} \cite[\S II.2.1.4]{blk} $D$ of $C(D)$, metrized with the uniform distance on the unit ball of $C(D)$. Since for every two distinct points $x_{\pm 1}\in D$ we can find a continuous function
    \begin{equation*}
      f:D\to [-1,1],\quad f(x_{\pm})=\pm 1. 
    \end{equation*}
    (e.g. by the {\it Tietze extension theorem}) \cite[Theorem 35.1]{mnk}, the distance between any two distinct points of $D$ is $2$. The topology acquired by $D$ as the metric space $\mathrm{hom}_{\cC^*_{c,1}}(C(D),\bC)$ is thus discrete.
  \end{itemize}
\end{example}

\Cref{ex:adjnotenrc} in fact illustrates a broader principle: as a $\cat{CMet}$-enriched category, $\cC^*_{c,1}$ is no more interesting than the plain category $\cC^*_{c,1}$:

\begin{proposition}\label{pr:discmet}
  For any two commutative $C^*$-algebras $A,B\in \cC^*_{c,1}$, the metric space
  \begin{equation*}
    \mathrm{hom}_{\cC^*_{c,1}}(A,B)\in \cat{CMet}
  \end{equation*}
  is discrete, with distance $2$ between any two distinct elements.
\end{proposition}
\begin{proof}
  The proof is essentially contained in \Cref{ex:adjnotenrc}. We have
  \begin{equation*}
    A\cong C(X),\quad B\cong C(Y)
  \end{equation*}
  for compact Hausdorff $X$ and $Y$, and distinct morphisms $\varphi\ne \psi:A\to B$ dualize to distinct continuous maps
  \begin{equation*}
    \varphi^*\ne \psi^*:Y\to X.
  \end{equation*}
  This means that
  \begin{equation*}
    x_+:=\varphi^*(y)\ne \psi^*(y)=:x_-,\quad\text{for some }y\in Y,
  \end{equation*}
  and a continuous norm-1 function 
  \begin{equation*}
    f:X\to [-1,1],\quad \varphi(x_{\pm}) = \pm 1
  \end{equation*}
  will be mapped by $\varphi$ and $\psi$ to norm-1 functions on $Y$ that are at least 2 apart. This shows that $d(\varphi,\psi)\ge 2$. On the other hand, because $\varphi$ and $\psi$ are contractions, we have
  \begin{equation*}
    d(\varphi,\psi) = \sup_{\|f\|\le 1}\|\varphi(f)-\psi(f)\|\le 2. 
  \end{equation*}
  This concludes the proof that the distance between two distinct morphisms between two commutative $C^*$-algebras is precisely 2.
\end{proof}

\begin{remark}
  Nothing like \Cref{pr:discmet} holds for the category $\cC^*_1$ of {\it arbitrary} (unital but not necessarily commutative) $C^*$-algebras: the identity automorphism of a non-commutative $C^*$-algebra $A$ will generally be approximable arbitrarily well by inner automorphisms
  \begin{equation*}
    A\ni a\mapsto uau^*\in A
  \end{equation*}
  for unitaries $u\in A$ close to $1$.
\end{remark}

\Cref{pr:discmet} has some consequences for the structure of $\cC^*_{c,1}$ as a $\cat{CMet}$-enriched category. First, recall the following notion from \cite[\S 3.7]{kly}.

\begin{definition}\label{def:enrc}
  A $\cV$-enriched category $\cC$ over a symmetric monoidal closed category $\cV$ is {\it $\cV$-tensored} (or {\it tensored over $\cV$}) if for objects $x\in \cV$ and $c\in \cC$ the $\cV$-functor
  \begin{equation*}
    \cV(x,\cC(c,-)):\cC\to \cV
  \end{equation*}
  is representable. In that case we denote the representing object by $x\otimes c$.
\end{definition}

\begin{proposition}\label{pr:cctens}
  The $\cat{CMet}$-enriched category $\cC^*_{c,1}$ is tensored over $\cat{CMet}$.
\end{proposition}
\begin{proof}
  It will be enough to prove that for $X\in \cat{CMet}$ and commutative $C^*$-algebras $A,B\in \cC^*_{c,1}$ we have natural identifications
  \begin{equation}\label{eq:xab}
    \cC^*_{c,1}(X\otimes A,B)\cong \cat{CMet}(X,\cC^*(A,B))
  \end{equation}
  of {\it sets} (for some appropriate $X\otimes A\in \cC^*_{c,1}$): \Cref{pr:discmet} will then provide the identification as {\it metric spaces}, since both sides are discrete metric spaces with distance 2 between any two distinct points.

  Appealing again to \Cref{pr:discmet} for the discrete metric structure of $\cC^*_{c,1}(A,B)$, the right-hand side of \Cref{eq:xab} (regarded as a set) is nothing but
  \begin{equation*}
    \cat{Set}(X/\sim,\ \cC^*(A,B)),
  \end{equation*}
  where `$\sim$' is the smallest equivalence relation on $X$ that identifies two points less than 2 apart. But then we can simply define the tensor product $X\otimes A$ as
  \begin{equation*}
    X\otimes A := (X/\sim)\otimes A = A^{\otimes (X/\sim)},
  \end{equation*}
  where
  \begin{itemize}
  \item the second `$\otimes$' denotes the tensored structure of $\cC^*_{c,1}$ over \cat{Set} (rather than \cat{CMet});
  \item and the last equality identifies that tensor product with the corresponding {\it copower} \cite[\S 3.7]{kly} of $A$, i.e. coproduct in $\cC^*_{c,1}$ of copies of $A$.
  \end{itemize}
  As explained, this concludes the argument.
\end{proof}

\section{$W^*$-algebras}\label{se:wast}

We make a few remarks on the category $\cW^*_1$ of {\it von Neumann (or $W^*$-)algebras} (\cite[\S III]{blk}, \cite[Definitions II.3.2 and III.3.1]{tak1}, etc.). These are the (unital) $C^*$-algebras that satisfy any of a number of mutually equivalent conditions:
\begin{itemize}
\item they are realizable as $C^*$-subalgebras $A\subseteq B(\cH)$ for some Hilbert space $\cH$ so that $A=A''$, the {\it double commutant} (i.e. the set of operators commuting with everything that centralizes $A$);
\item or $A\subseteq B(\cH)$ is closed in any of the six ``weak'' operator topologies of \cite[\S II.2]{tak1} (weak, $\sigma$-weak, strong, $\sigma$-strong, strong$^*$, $\sigma$-strong$^*$) \cite[Part I, \S 3.4, Theorem 2]{dixw};
\item or $A$ is a $C^*$-algebra that is also a dual Banach space (whose predual is then automatically unique) \cite[Theorem III.3.5 and Corollary III.3.9]{tak1}.
\end{itemize}

The morphisms in $\cW^*_1$ are the unital {\it normal} \cite[Proposition III.2.2.2]{blk} $C^*$-morphisms, i.e. those that are continuous for the $\sigma$-weak topologies induced by the realizations of the von Neumann algebras as duals.

The resulting category $\cW^*_1$ is {\it almost} that studied in \cite{gch}, the difference being that the latter source considers possibly non-unital morphisms: this unitality is what the `$1$' subscript on `$\cW^*_1$' is meant to remind the reader of. In fact, it will be convenient to also, on occasion, refer to the category of \cite{gch}: von Neumann algebras with normal, not-necessarily-unital morphisms (if only to contrast several results and phenomena). The symbol for this latter category will be $\cW^*$.

Despite the difference between $\cW^*$ (central to \cite{gch}) and $\cW^*_1$ (studied here), the arguments of loc.cit. do go through with only minor modifications to provide arbitrary coproducts in $\cW^*_1$ by assembling together finite coproducts \cite[\S 5]{gch} and filtered colimits \cite[\S 7]{gch}. Coequalizers are also easily constructed \cite[\S 2, p.43]{gch}: for parallel $\cW^*_1$-morphisms
\begin{equation*}
  f,g:A\to B
\end{equation*}
simply quotient $B$ by the $\sigma$-weakly-closed ideal generated by
\begin{equation*}
  \{f(a)-g(a)\ |\ a\in A\}. 
\end{equation*}
All of this is to show that, per \cite[\S V.2, Theorem 1]{mcl}, we have (\cite[Proposition 5.7]{krn-qcoll})
\begin{proposition}\label{pr:wccmpl}
  The category $\cW^*_1$ is cocomplete.  \qedhere
\end{proposition}

\subsection{(Non-)presentability}\label{subse:npres}

One occasionally finds claims or suggestions to the effect that $\cW^*_1$ is locally presentable. This is stated outright in \cite{nlab-vna}, for instance, and claimed for {\it $i\bR$-graded} von Neumann algebras in \cite[paragraph following Definition 3.4]{pvl-lp} (the grading will not make much of a difference; see \Cref{subse:grd}). We prove here the following strong negation of local presentability.

\begin{theorem}\label{th:nopres}
  The only presentable objects in the category $\cW^*_1$ of von Neumann algebras with normal unital morphisms are the zero algebra and $\bC$.
\end{theorem}

Before turning to the proof, note the obvious consequence:

\begin{corollary}\label{cor:npres}
  The category $\cW^*_1$ of von Neumann algebras is not locally presentable.  \qedhere
\end{corollary}

\pf{th:nopres}
\begin{th:nopres}
  We will argue that a von Neumann algebra $A$ of dimension $\ge 2$ admits, for arbitrarily large regular cardinals $\kappa$, a morphism
  \begin{equation}\label{eq:nfact}
    A\to\left(\text{$\kappa$-directed }\varinjlim_i A\right)\text{ not factoring through any }A_i.
  \end{equation}
  We will make a series of simplifications to this end.

  {\bf Step 1: reducing to Cartesian factors.} If $A$ admits a morphism \Cref{eq:nfact} so does $B\times A$, by simply applying the functor $B\times-$ (Cartesian product with $B$) to \Cref{eq:nfact}: indeed, that functor is easily seen to preserve directed colimits, so we obtain a morphism
  \begin{equation*}
    B\times A\to B\times \varinjlim A_i\cong \varinjlim (B\times A_i)
  \end{equation*}
  with the desired properties.

  {\bf Step 2: reducing to (certain) tensorands.} Suppose, once more, that we have a morphism \Cref{eq:nfact}, and consider the {\it maximal} (or {\it universal}) tensor product $B\widehat{\otimes }A$. This is the construction denoted by $B\overset{\mu}{\otimes} A$ after \cite[Proposition 8.2]{gch}, and is the universal codomain of two morphisms out of $A$ and $B$ with commuting ranges.

  The universality property of the `$\widehat{\otimes}$' construction makes it clear that the functor $B\widehat{\otimes}-$ preserves directed colimits, so we once again obtain a morphism
  \begin{equation}\label{eq:otimesnfact}
    B\widehat{\otimes} A\to B\widehat{\otimes} \varinjlim A_i\cong \varinjlim (B\widehat{\otimes} A_i).
  \end{equation}
  That it cannot factor through any of the individual $B\widehat{\otimes}A_i$ will follow, in the particular cases we consider:
  \begin{itemize}
  \item the connecting morphisms in the directed diagram producing $\varinjlim_i A_i$ will {\it injective} morphisms, as will \Cref{eq:nfact};
  \item $A$ will be {\it flat} (`plat' in \cite[\S 5]{gch}), in the sense that the canonical surjections
    \begin{equation*}
      B\widehat{\otimes} A\to B\overline{\otimes}A,\ B\in \cW^*_1
    \end{equation*}
    onto the {\it spatial} tensor product (the `$\overset{c}{\otimes}$' of \cite[discussion preceding Lemme 8.1]{gch}, `$\overline{\otimes}$' of \cite[\S III.1.5.4]{blk} or \cite[Definition IV.5.1]{tak1}, etc.) are all isomorphisms;
  \item whence applying the $\widehat{\otimes}$-to-$\overline{\otimes}$ surjection to \Cref{eq:otimesnfact} we obtain
    \begin{equation*}
      B\widehat{\otimes} A \cong B\overline{\otimes} A \subseteq B\overline{\otimes} (\varinjlim A_i),
    \end{equation*}
    using the fact that `$\overline{\otimes}$' produces injections when applied to injections \cite[Proposition 8.1]{gch}.
  \item We can then argue that $B\overline{\otimes}A$ is not contained in any $B\overline{\otimes}A_i$ by observing that \cite[Corollary IV.5.10]{tak1}
    \begin{equation*}
      (B\overline{\otimes}A)\cap (B\overline{\otimes}A_i) = B\overline{\otimes}(A\cap A_i)\subset B\overline{\otimes}A;
    \end{equation*}
  \item and the properness of the latter inclusion follows from that of $A\cap A_i\subset A$ by applying some functional (\cite[equation (1) following Definition IV.5.1]{tak1}) of the form $\varphi_B\overline{\otimes}\varphi_A$ where $\varphi_B\in B_*$ and $\varphi_A\in A_*$ vanishes on $A\cap A_i$ but not on $A$.
  \end{itemize}

  {\bf Step 3: reducing the problem to abelian von Neumann algebras and matrix algebras.} For an arbitrary von Neumann algebra $A$ we have its canonical direct-product decomposition \cite[Theorem V.1.19]{tak1}
  \begin{equation*}
    A\cong A_{I}\times A_{II_1}\times A_{II_{\infty}}\times A_{III}
  \end{equation*}
  into factors of {\it types} $I$, $II_1$, etc. Step 1 above then allows us to focus on the individual single-type Cartesian factors.

  Those of types $II$ or $III$ always decompose as $A\cong A'\otimes M_2$ \cite[Proposition V.1.35]{tak1} and $M_2$ is flat \cite[Proposition 8.6]{gch} (the term `discret' used there is synonymous to `type-$I$'), so Step 2 reduces the entire discussion to $M_2$.

  On the other hand, a type-I von Neumann algebra decomposes \cite[Theorem V.1.27]{tak1} as a product of the form
  \begin{equation*}
    \prod_{\alpha} \left(C_{\alpha}\otimes B(\cH_{\alpha})\right)
  \end{equation*}
  for various cardinals $\alpha$, where $C_{\alpha}$ is abelian and $\cH_{\alpha}$ is an $\alpha$-dimensional Hilbert space.

  If we have any non-vanishing factors for $\alpha\ge 2$ we can again peel off a matrix-algebra tensorand from $B(\cH_{\alpha})$ and proceed as before. On the other hand, the $\alpha=1$ factor reduces to an abelian von Neumann algebra.

  Having thus reduced the discussion to von Neumann algebras that are either
  \begin{itemize}
  \item abelian of dimension $\ge 2$;
  \item or matrix algebras $M_n$, $n\ge 2$,
  \end{itemize}
  we relegate those two cases to two separate results: \Cref{pr:ab} and \Cref{pr:mtrx} respectively.
\end{th:nopres}

\begin{remark}\label{re:sametens}
  A point of clarification is perhaps in order regarding the maximal tensor product denoted above by $\widehat{\otimes}$. As noted, it coincides with $\overset{\mu}{\otimes}$ of \cite[\S 8]{gch}; this might appear surprising at first, given that
  \begin{itemize}
  \item the morphisms in $\cW^*_1$ and $\cW^*$ are different;
  \item while the tensor products $\widehat{\otimes}$ and $\overset{\mu}{\otimes}$ supposedly have the same universality property, each phrased in the respective category: $A\bullet B$ (for either choice of `$\bullet$') is supposed to be the universal recipient of two mutually commuting morphisms from $A$ and $B$. 
  \end{itemize}
  The problem seems to be that the universality property of $\overset{\mu}{\otimes}$ is (or at least appears to us to be) misstated in \cite[Proposition 8.2]{gch}.

  Specifically, that universality property is stated as above, via couples of 
  commuting-image morphisms, but by construction (reverting to the notation in 
  \cite[\S 8]{gch}) the two morphisms $A_i\to A_1\overset{\mu}{\otimes}{A_2}$, 
  $i=1,2$ send the respective units $e_{A_i}$ to the same element 
  $g(e_{A_1}\otimes e_{A_2})$. It follows that the 
  $A_1\overset{\mu}{\otimes}A_2$ construction only classifies those pairs of 
  morphisms out of $A_i$, $i=1,2$ which
  \begin{itemize}
  \item have commuting images, as discussed previously;
  \item and {\it additionally}, send the units of $A_i$, $i=1,2$ onto the same projection of the codomain.
  \end{itemize}
  With this caveat, it is no longer surprising that the two tensor products coincide: the categories differ, but so does the (phrasing of the) universality property. 
\end{remark}

We need the following simple observation on realizing $\ell^{\infty}$ von Neumann algebras as sufficiently-directed colimits.

\begin{lemma}\label{le:discr-colim}
  Let $\kappa$ be a regular cardinal and denote by $\kappa_*$ the set $\kappa\sqcup\{*\}$ (i.e. $\kappa$ with an additional distinguished point). For each subset $S\subset\kappa$ of cardinality $<\kappa$, denote by
  \begin{equation*}
    C_S\subset \ell^{\infty}(\kappa_*)
  \end{equation*}
  the unital von Neumann subalgebra generated by the minimal projections associated to the elements of $S$. The canonical morphism
  \begin{equation}\label{eq:ellcolim}
    \varinjlim_S C_S\to \ell^{\infty}(\kappa_*),\quad\text{sets $S$ ordered by inclusion}
  \end{equation}
  is an isomorphism, thus realizing $\ell^{\infty}(\kappa_*)$ as a $\kappa$-directed colimit in $\cW^*_1$.
\end{lemma}
\begin{proof}
  That the colimit is $\kappa$-directed follows from the regularity of $\kappa$: a union of fewer than $\kappa$ sets of cardinality $<\kappa$ again has cardinality $<\kappa$.

  As for \Cref{eq:ellcolim} being an isomorphism, we can check this by examining the functors $\cW^*_1\to \cat{Set}$ represented by the two objects. First, note that
  \begin{equation*}
    \mathrm{hom}(C_S,-):\cW^*_1\to \cat{Set}
  \end{equation*}
  is the functor that picks out a projection (the image of the characteristic function $\chi_S\in C_S$) and a partition of that projection into an $S$-indexed set of mutually-orthogonal projections. Passing to the limit,
  \begin{equation*}
    \varprojlim_S \mathrm{hom}(C_S,-) \cong \mathrm{hom}\left(\varinjlim_S C_S,\ -\right)
  \end{equation*}
  is the functor that picks out a projection $p$ (in whatever von Neumann algebra the functor is being applied to) and decomposes it as a $\kappa$-indexed partition of mutually-orthogonal projections. But this is also the functor represented by $\ell^{\infty}(\kappa_*)$: the unit of the non-unital von Neumann subalgebra
  \begin{equation*}
    \ell^{\infty}(\kappa)\subset \ell^{\infty}(\kappa_*)
  \end{equation*}
  gets mapped to $p$, and the characteristic function $\chi_{\{*\}}$ of the leftover singleton $\{*\}=\kappa_*\setminus \kappa$ gets mapped to $1-p$.
\end{proof}

\begin{remark}
  In particular, \Cref{le:discr-colim} gives an example (or rather a family of examples) of arbitrarily-directed families
  \begin{equation*}
    A_S\subset A
  \end{equation*}
  of von Neumann subalgebras for which $\varinjlim_S A_S\to A$ is not one-to-one: simply take $A=\ell^{\infty}(\kappa)$ (not $\kappa_*$!) and $A_S$ to be the unital von Neumann subalgebra generated by the minimal projections attached to elements of $S\subset \kappa$ for $|S|<\kappa$. \Cref{le:discr-colim} implies that
  \begin{equation}\label{eq:limasa}
    \varinjlim_S A_S\to A
  \end{equation}
  can be identified with the first-factor projection
  \begin{equation*}
    \ell^{\infty}(\kappa_*)\cong A\times \bC\to A\cong \ell^{\infty}(\kappa). 
  \end{equation*}  
  
  For \Cref{le:discr-colim} to go through, it is crucial that we work with the category of $W^*$-algebras with {\it unital} morphisms. By contrast, computing the colimit in the category studied in \cite{gch}, of $W^*$-algebras with normal, possibly non-unital morphisms, the more ``reasonable'' morphism
  \begin{equation*}
    \varinjlim_S C_S\to \ell^{\infty}(\kappa)
  \end{equation*}
  is an isomorphism: both objects, in that case, represent the functor
  \begin{equation*}
    \left(\text{$W^*$-algebra A}\right)\mapsto \left(\text{$\kappa$-partitioned projections in $A$}\right)\in\cat{Set}.
  \end{equation*}
  For the non-unital $W^*$ category \cite[examples following Remarque 7.2]{gch} similarly illustrate the phenomenon of morphisms \Cref{eq:limasa} out of directed colimits failing to be one-to-one despite $A_S\subseteq A$ being so.
\end{remark}

\begin{proposition}\label{pr:ab}
  An abelian von Neumann algebra of dimension $\ge 2$ is not $\kappa$-presentable in $\cW^*_1$ for any regular cardinal $\kappa$.
\end{proposition}
\begin{proof}
  Isolating out the minimal projections, an abelian von Neumann algebra $A_1$ decomposes as $A_1\times A_2$ with $A_1\cong \ell^{\infty}(S)$ for some set $S$ and $A_2$ {\it non-atomic}, in the sense that it has no minimal non-zero projections (the term is measure-theoretic \cite[\S 40]{hlm-mt}; {\it diffuse} is a synonym: \cite[\S III.4.8.8]{blk}, \cite[\S 7.11.16]{ped-aut}, \cite[\S 1]{oz-sld}, etc.).

  The diffuse factor $A_2$ is of the form $L^{\infty}(X,\mu)$ for a non-atomic \cite[\S 40]{hlm-mt} positive measure $\mu$ on a locally compact Hausdorff space $X$ \cite[Part I, \S 7.3, Theorem 1]{dixw}. By general measure-space structure theory (\cite[Theorems 1 and 2]{mhrm} or \cite[Theorems 2 and 3]{znk}) we can decompose $A_2$ as a product of copies of $L^{\infty}(\bR,\mu_{\mathrm{Lebesgue}})$, and in turn the latter decomposes as 
  \begin{equation*}
    L^{\infty}(\bR,\mu)\cong L^{\infty}(\bR,\mu)^2\cong L^{\infty}(\bR,\mu)\overline{\otimes} \bC^2. 
  \end{equation*}
  By the reduction steps in the proof of \Cref{th:nopres}, the problem then boils down to $\bC^2$.

  All of this is available provided the diffuse factor $A_2$ doesn't vanish. If it does, $A_1$ is $\ell^{\infty}(S)$ for $|S|\ge 2$, it surjects onto $\bC^2$, and again the problem reduces to the latter provided we map it into a directed colimit of embeddings.

  All in all, we can now focus on $\bC^2$. To that end, simply note the embedding $\bC^2\subset \ell^{\infty}(\kappa_*)$ (notation as in \Cref{le:discr-colim}) that sends one of the two minimal projections, say $p\in \bC^2$, to the characteristic function
  \begin{equation*}
    \chi_{\kappa}\in \ell^{\infty}(\kappa_*) = \ell^{\infty}(\kappa\sqcup\{*\}). 
  \end{equation*}
  This is an embedding into a $\kappa$-directed colimit by \Cref{le:discr-colim}, but cannot factor through any of the $C_S\subset \ell^{\infty}(\kappa_*)$, $|S|<\kappa$ because none $C^*$-subalgebras of those contain $\chi_{\kappa}$.
\end{proof}

\begin{proposition}\label{pr:mtrx}
  A matrix algebra $M_n=M_n(\bC)$, $n\ge 2$ is not $\kappa$-presentable in $\cW^*_1$ for any regular cardinal $\kappa$.
\end{proposition}
\begin{proof}
  We specialize to $n=2$, in order to fix ideas and ease the notational load of the proof; the general argument does not pose substantial additional difficulties.

  Consider a $\kappa$-dimensional Hilbert space $\cH$, with a fixed orthonormal basis
  \begin{equation*}
    e_{s,i}\in \cH,\ s\in \kappa,\ i=0,1.
  \end{equation*}
  The extra $0,1$ indices are there to aid in working with our choice of $n=2$; in general the index $i$ would run from $0$ to $n-1$. Fix, for each $s\in\kappa$, a partial isometry $u_s$ that implements an equivalence between the projections $p_{s,i}$ onto $\bC e_{s,i}$, $i=0,1$:
  \begin{equation*}
    u_s^*u_s = p_{s,0},\quad u_s u_s^* = p_{s,1}.
  \end{equation*}
  For a subset $S\subset \kappa$ of strictly smaller cardinality write
  \begin{itemize}
  \item $\cH_S$ for the closed span of $\{e_{s,i}\ |\ s\in S,\ i=0,1\}$;
  \item $M_{2,S^c}$ (the `c' superscripts stands for `complement') for the copy of $M_2$ spanned by
    \begin{equation*}
      \sum_{s\not\in S}p_{s,i},\ i=0,1,\quad \sum_{s\not\in S}u_s\quad\text{and}\quad \sum_{s\not\in S}u_s^*
    \end{equation*}
    (this is a non-unital $W^*$-subalgebra of $B(\cH)$);
  \item and $B_S$ for the direct sum
    \begin{equation*}
      B_S:=B(\cH_S)\oplus M_{2,S^c}. 
    \end{equation*}
  \end{itemize}
  A number of observations are immediate:
  \begin{enumerate}[(1)]
  \item The map $S\mapsto B_S$ is monotone for inclusion, and makes $B_S\subset B(\cH)$ into a $\kappa$-directed family of von Neumann subalgebras.
  \item Every $B_S$ is {\it 2-homogeneous} (a slight extension of the language used in \cite[\S 5.5.6]{ped-aut}, for instance): there is a projection $p\in B_S$ equivalent to $1-p$ in the sense that
    \begin{equation*}
      u^*u = p,\ uu^*=1-p
    \end{equation*}
    for some $u\in B_S$ or, equivalently, there is a unital $W^*$-morphism $M_2\to B_S$.
  \item The $B_S$ jointly generate $B(\cH)$ as a von Neumann algebra, i.e.
    \begin{equation}\label{eq:surjbh}
      B:=\varinjlim_S B_S\to B(\cH)
    \end{equation}
    is onto.
  \end{enumerate}
  Consider, now, a morphism $M_2\to B(\cH)$ that sends the projection
  \begin{equation*}
    p:=
    \begin{pmatrix}
      1&0\\
      0&0
    \end{pmatrix}
  \end{equation*}
  to, say, $p_{T,0}:=\sum_{s\in T}p_{s,0}$ for some subset
  \begin{equation*}
    T\subset\kappa,\quad |T| = \kappa = |T^c|. 
  \end{equation*}
  The surjection \Cref{eq:surjbh} is, like all $W^*$-algebra surjections \cite[Part I, \S 3.4, Corollary 3]{dixw}, identifiable with
  \begin{equation*}
    B\ni b\mapsto bz\in Bz
  \end{equation*}
  for some central projection $z\in B$. The 2-homogeneity of the $B_S$ implies that of $B(1-z)$, so the morphism $M_2\to B(\cH)$ lifts along \Cref{eq:surjbh} to
  \begin{equation}\label{eq:m2b}
    M_2\to B\cong Bz\times B(1-z)\to Bz\cong B(\cH)
  \end{equation}
  by preserving the $Bz$ component of the morphism we set out with and supplementing it with an arbitrary $M_2\to B(1-z)$.
  
  Because $p_{T,0}$ is not contained in any of the $B_S\subset B(\cH)$, \Cref{eq:m2b} cannot factor through any $B_S\subset \varinjlim_S B_S=B$. We thus have a morphism into a $\kappa$-directed colimit which doesn't factor through any of the colimit constituents: precisely what we sought. 
\end{proof}

\subsection{Additional structure: graded $W^*$-algebras}\label{subse:grd}

To return to the motivating discussion, we recall briefly the framework relevant to \cite{pvl-lp}. That paper works extensively with what it calls {\it {\bf I}-graded von Neumann algebras}, where
\begin{itemize}
\item {\bf I} is the imaginary line $(i\bR,+)$ regarded as a locally compact abelian group;
\item and being {\bf I}-graded is a specialization of the following notion (\cite[Definition 3.1]{pvl-lp}).
\end{itemize}

\begin{definition}\label{def:grd}
  For a locally compact abelian group $\bG$ (always Hausdorff, for us), a {\it $\bG$-graded von Neumann algebra} is a von Neumann algebra $M$ equipped with an action by the {\it Pontryagin dual}
  \begin{equation*}
    \widehat{\bG} := \{\text{continuous group morphisms }\bG\to \bS^1\}
  \end{equation*}
  by $W^*$-algebra automorphisms, such that for every $m\in M$ the map
  \begin{equation*}
    \bG\ni g\mapsto gm\in M
  \end{equation*}
  is continuous from $\bG$ with its locally-compact topology to $M$ equipped with its weak$^*$ topology.

  We write $\cW^*_{1,\bG}$ for the category of $\bG$-graded von Neumann algebras, with morphisms being those in $\cW^*_1$ which intertwine the $\widehat{\bG}$-actions.
\end{definition}

We remind the reader that
\begin{itemize}
\item $\bG\mapsto \widehat{\bG}$ is a {\it duality} on the category of locally compact (Hausdorff) abelian groups, i.e. an involutive contravariant equivalence (as follows, for instance, from \cite[Theorems 7.7 and 7.63]{hm}).  
\item The continuity requirement for the $\widehat{\bG}$-action on $M$ imposed in \Cref{def:grd} is the standard one in the theory of locally-compact-group actions on von Neumann algebras: see e.g. \cite[\S X.1, Definition 1.1 and Proposition 1.2]{tak2}.
\end{itemize}

We can now directly address the claim made in passing in \cite[paragraph following Definition 3.4]{pvl-lp} that the category of {\bf I}-graded von Neumann algebras is locally presentable. The following result, generalizing \Cref{cor:npres}, shows that this cannot be the case.

\begin{proposition}\label{pr:grnpres}
  The category $\cW^*_{1,\bG}$ of $\bG$-graded von Neumann algebras is not locally presentable for any locally compact abelian group $\bG$.
\end{proposition}
\begin{proof}
  We can regard any von Neumann algebra as carrying the {\it trivial} $\bG$-grading, i.e. the trivial $\widehat{\bG}$-action; this gives a (fully faithful) functor $F:\cW^*_1\to \cW^*_{1,\bG}$.

  On the other hand, for every $\gamma\in \bG$ and $M\in \cW^*_{1,\bG}$ equipped with a $\widehat{\bG}$-action
  \begin{equation*}
    \triangleright:\widehat{\bG}\times M\to M,
  \end{equation*}
  we can consider the {\it degree-$\gamma$ component} of $M$, defined by
  \begin{equation*}
    M_{\gamma}:=\{m\in M\ |\ p\triangleright m = \gamma(p)m,\ \forall p\in \widehat{\bG}\}.
  \end{equation*}
  In particular, selecting the {\it $\widehat{\bG}$-invariants}
  \begin{equation*}
    M^{\widehat{\bG}}:=M_1
    = \{m\in M\ |\ p\triangleright m = m,\ \forall p\in \widehat{\bG}\}
    \subseteq M
  \end{equation*}
  gives a functor
  \begin{equation*}
    \cW^*_{1,\bG}\ni M\stackrel{G}{\longmapsto} M^{\widehat{\bG}}\in M
  \end{equation*}
  that is right adjoint to $F$.

  Since $F$ is a left adjoint, it is cocontinuous. This means that any morphism
  \begin{equation*}
    M\to \varinjlim_i M_i
  \end{equation*}
  in $\cW^*_1$ witnessing that $M$ is {\it not} $\kappa$-presentable (for some regular cardinal $\kappa$) can be regarded as such a diagram in $\cW^*_{1,\bG}$ instead.

  It follows from \Cref{th:nopres} that the only objects in $\cW^*_{1,\bG}$ carrying a trivial $\bG$-grading that are presentable (for any cardinal) are $\{0\}$ and $\bC$. This finishes the proof, since in a locally presentable category every object is presentable for {\it some} cardinal \cite[Remark 1.30 (1)]{ar}.
\end{proof}

\subsection{Back to monadic forgetful functors}\label{subse:wmnd}

It is natural to consider, at this stage, analogues of \Cref{th:cbanalg} and \Cref{cor:frgmnd} in the $W^*$setup.

\begin{theorem}\label{th:wtocmnd}
  The forgetful functor $G:\cW^*_1\to \cC^*_1$ satisfies the Crude Tripleability Theorem, so is in particular monadic.
\end{theorem}
\begin{proof}
  We verify the conditions listed in \Cref{def:tt}.
  \begin{enumerate}[(1)]
  \item The embedding $\cW^*_1\to \cC^*_1$ does indeed have a left adjoint, namely the double-dual functor $A\mapsto A^{**}$ \cite[Theorem III.2.4]{tak1}. 
  \item Isomorphism-reflection is well known, and in fact much more can be said: a purely algebraic isomorphism of $*$-algebras between two $W^*$-algebras is norm-continuous and normal \cite[Theorem 3]{tomi1}.
  \item $\cW^*_1$ being cocomplete (\Cref{pr:wccmpl}), {\it all} coequalizers exist (not just reflexive ones). 
  \end{enumerate}
  It thus remains to check that $G$ preserves reflexive coequalizers. Consider, then, a reflexive pair $\partial_i:A\to B$, $i=0,1$ in $\cW^*_1$, with a common right inverse $t:B\to A$:
  \begin{equation}\label{eq:cominv}
    \partial_1 t = \id_B = \partial_0 t. 
  \end{equation}
  The idempotent endomorphisms
  \begin{equation*}
    e_i := t\partial_i,\ i=0,1 
  \end{equation*}
  of $A$ satisfy
  \begin{equation}\label{eq:rcpr}
    e_0 e_1 = e_1,\quad e_1 e_0=e_0
  \end{equation}
  as a consequence of \Cref{eq:cominv}. Because $t$ is one-to-one we have $\ker e_i = \ker\partial_i$, and furthermore these kernels are the ideals generated by central projections $z_i\in A$ \cite[Part I, \S 3.3, Corollary 3]{dixw}. \Cref{eq:rcpr} says that each $e_i$ is faithful on $(1-z_j)A$, $j\ne i$, and hence each of the projections $z_i$, $i=0,1$ dominates the other. But this means that
  \begin{equation*}
    z_1 = z_0 =: z,\quad\text{so that }\ker\partial_1 = \ker\partial_0 = zA.
  \end{equation*}
  The conclusion, then, is that
  \begin{itemize}
  \item the maps $\partial_i$, $i=0,1$ both annihilate $zA$;
  \item and they agree on the complementary subspace $(1-z)A$: for each $i=0,1$ and $a\in (1-z)A$, the elements
    \begin{equation*}
      a\quad\text{and}\quad (1-z) (t\partial_i)(a)
    \end{equation*}
    are both mapped to the same element $\partial_i(a)\in B$. Since
    \begin{itemize}
    \item $\partial_i$ is injective on $(1-z)A$;
    \item and $(1-z)t(\cdot):B\to (1-z)A$ is injective (being a right inverse to the restrictions $\partial_i|_{(1-z)A}$),
    \end{itemize}
    we have $\partial_1(a)=\partial_1(a)$.
  \end{itemize}
  This means that the coequalizer of the pair $(\partial_i)_{i=0,1}$ is precisely $\id_B:B\to B$, clearly preserved by the forgetful functor.
\end{proof}

\begin{remark}\label{re:notallcoeq}
  In the latter part of the proof of \Cref{th:wtocmnd}, it was crucial to restrict the discussion to {\it reflexive}-coequalizer preservation: the forgetful functor $G:\cW^*_1\to \cC^*_1$ does not preserve {\it arbitrary} coequalizers.

  Consider, for instance, the pair $\partial_i:B(\cH)\to B(\cH)$, $i=0,1$ for a countably-infinite-dimensional Hilbert space $\cH$, where $\partial_0=\id$ and $\partial_1$ is conjugation by the unitary
  \begin{equation*}
    u:=1-2p,\quad p\in B(\cH)\ \text{a projection of rank 1}.
  \end{equation*}
  Because $u$ differs from the identity by a {\it compact} \cite[\S I.8]{blk} operator, every difference
  \begin{equation*}
    \partial_1(a)-\partial_0(a),\ a\in A
  \end{equation*}
  belongs to the ideal \cite[I.8.1.2]{blk} $K(\cH)\subset B(\cH)$ of compact operators.

  Now, on the one hand, because $\dim\cH=\aleph_0$ (as a Hilbert space), $K(\cH)$ is the unique proper non-zero ideal of $B(\cH)$ \cite[Proposition III.1.7.11]{blk} so the $C^*$ cokernel of $(\partial_i)_{i=0,1}$ is the surjection
  \begin{equation*}
    B(\cH)\to B(\cH)/K(\cH)
  \end{equation*}
  onto the {\it Calkin algebra} \cite[I.8.2]{blk} of $\cH$.

  On the other hand, $B(\cH)$ is a {\it factor} \cite[\S I.9.1.5]{blk} (its center is $\bC$), so it has no proper, non-zero weak$^*$-closed ideals \cite[Part I, \S 3.3, Corollary 3]{dixw}. This means that the cokernel of $(\partial_i)_{i=0,1}$ in $\cW^*_1$ is the zero algebra.
\end{remark}

\begin{corollary}\label{cor:wallmnd}
  The forgetful functors from $\cW^*_1$ to any of the categories $\cat{BanAlg}^*_1$, $\cat{BanAlg}_1$ or $\cat{Ban}$ are all monadic.
\end{corollary}
\begin{proof}
  As in the proof of \Cref{cor:frgmnd}: said functors decompose as 
  \begin{itemize}
  \item $\cW^*_1\to \cC^*_1$, which is CTT \Cref{th:wtocmnd};
  \item followed by forgetful functors from $\cC^*_1$ to the respective categories, all of which are monadic \Cref{cor:frgmnd}.
  \end{itemize}
  \Cref{le:cttptt} thus applies to conclude. 
\end{proof}



\addcontentsline{toc}{section}{References}

\Addresses

\end{document}